\documentclass[preprint,12pt]{elsarticle}
\usepackage{amsthm,amsmath,amssymb}
\usepackage[colorlinks=true,citecolor=black,linkcolor=black,urlcolor=blue]{hyperref}
\usepackage{graphicx}
\usepackage{float}
\usepackage{epstopdf}
\usepackage{epsfig}
\usepackage{extarrows,chngpage,array,float,eqnarray}     

\journal{~}

\theoremstyle{plain}
\newtheorem{theorem}{Theorem}[section]
\newtheorem{lemma}[theorem]{Lemma}
\newtheorem{corollary}[theorem]{Corollary}

\newtheorem{observation}[theorem]{Observation}

\theoremstyle{definition}
\newtheorem{definition}[theorem]{Definition}

\newtheorem{question}[theorem]{Question}

\theoremstyle{remark}

\title{On the connected coalition number}
\author{Xiaxia Guan$^a$,  Maoqun Wang$^b$\footnote{Corresponding author.} \\
  \small $^a$Department of Mathematics, Taiyuan University of
Technology, Taiyuan-030024, PR China \\
  \small $^b$School of Mathematics and Information Sciences, Yantai University, Yantai-264005, PR China\\
 \small \emph{Email}: gxx0544@126.com; wangmaoqun@ytu.edu.cn}
\begin{document}
\begin{abstract}
For a graph $G=(V,E)$, a pair of vertex disjoint sets $A_{1}$ and $A_{2}$ form a connected coalition of $G$, if $A_{1}\cup A_{2}$ is a connected dominating set, but neither $A_{1}$ nor $A_{2}$ is a connected dominating set. A connected coalition partition of $G$ is a partition $\Phi$ of $V(G)$ such that each set in $\Phi$ either consists of only a singe vertex with the degree $|V(G)|-1$, or forms a connected coalition of $G$ with another set in $\Phi$. The connected coalition number of $G$, denoted by $CC(G)$, is the largest possible size of a connected coalition partition of $G$. In this paper, we characterize graphs that satisfy $CC(G)=2$. Moreover, we obtain the connected coalition number for unicycle graphs and for the corona product and join of two graphs. Finally, we give a lower bound on the connected coalition number of the Cartesian product and the lexicographic product of two graphs. 
\end{abstract}


\maketitle
\section{Introduction}
\noindent

Let $G$ be a graph. We denote by $V(G)$ and $E(G)$ the vertex set and edge set of $G$, respectively, and call $|V(G)|$ the order of $G$.  A {\it neighbour} of a vertex $v$ is a vertex adjacent to $v$. The {\it degree} of a vertex $v\in V$, denoted by $deg(v)$, is the number of its neighborhoods. A vertex with degree $|V(G)|-1$ in a graph $G$ is called a \emph{full vertex}. A vertex $v$ in $G$ is referred to as a \emph{pendant vertex} if $deg(v)=1$. For a vertex subset $S\subseteq V$, the subgraph induced by $S$, denoted by $G[S]$, is the subgraph whose vertex set is $S$ and whose edge set consists of all edges of $G$ which have both ends in $S$. The subgraph $G-S$ is the subgraph obtained by removing all vertices in $S$ and removing all edges incident with some vertex in $S$ from the graph $G$.

Many questions in combinatorics can be described as a certain type of domination problems in graphs. There is a vast literature on the various domination, see for instance  five fundamental books \cite{Du,Haynes1,Haynes2,Haynes3,Henning} and two surveys \cite{Goddard,Henning0}. In this paper, we study the connected coalition number of graphs, introduced recently by Alikhani, Bakhshesh, Golmohammadi and Konstantinova \cite{Alikhani},
similar to the coalition number. We only consider simple and finite graphs throughout this paper. Definitions which are not given here may be found in \cite{bondy}. Cockayne and Hedetniemi \cite{Cockayne} defined the domatic number of a graph. Later, the connected domatic number of a graph is introduced by Zelinka \cite{Zelinka}.

\begin{definition}
 Let $G$ be a graph. A vertex subset $S\subseteq V(G)$ is called \emph{a dominating set} of $G$, if  for each vertex $v\in V(G)\backslash S$, there exists at least one vertex $u\in S$ with $uv\in E(G)$. A vertex subset $S$ is called \emph{a connected dominating set} of $G$, if $S$ is a dominating set and $G[S]$ is connected. A \emph{connected domatic partition} of $G$ is a partition of $V(G)$ into connected dominating sets. The \emph{connected domatic number} of $G$, denoted by $d_{c}(G)$, is the maximum size of a connected domatic partition in $G$.
\end{definition}

We refer the readers to \cite{Hartnell,Zelinka1,Zelinka2,Zelinka} for more details and results on the domatic number and the connected domatic number of a graph. Haynes et al. \cite{Haynes} first introduced the concept of coalitions and coalition partitions in the field of graph theory. Later, the coalition number of some families of graphs is researched, see \cite{Alikhani2,Bakhshesh,Haynes4,Haynes5}. In 2022, Alikhani et al. \cite{Alikhani1} introduced the concept of total coalitions of a graph. In 2023, Bar\'{a}t and Bl\'{a}zsik \cite{Barat} obtained a general sharp upper bound on the total coalition number as a function of the maximum degree. Recently, Alikhani et al. \cite{Alikhani} introduced the concept of connected coalitions and connected coalition partitions in a graph.

\begin{definition}
 Let $G$ be a graph. A pair of vertex disjoint sets $A_{1}$ and $A_{2}$ form a {\it connected coalition} of $G$, if $A_{1}\cup A_{2}$ is a connected dominating set, but neither $A_{1}$ nor $A_{2}$ is a connected dominating set. A partition $\Phi=\{A_{1},A_{2},\ldots,A_{k}\}$ of $V(G)$ is called a \emph{connected coalition partition} of $G$, if for each set $A_{i}\in \Phi$, either $A_{i}=\{v\}$ for some full vertex $v$ of $G$, or $A_{i}$ and $A_{j}$ form a connected coalition of $G$ for another set $A_{j}\in\Phi$. The \emph{connected coalition number} of a graph $G$, denoted by $CC(G)$, is the maximum cardinality of a connected coalition partition in $G$. For a connected coalition partition $\Phi$ of $G$, we say that $\Phi$ is a $CC(G)$-\emph{partition} if $|\Phi|=CC(G)$.
\end{definition}

Clearly, the connected coalition number of a graph is at most the number of vertices. This upper bound can be obtain for complete graphs and complete bipartite graphs $K_{m,n}$ with $2\leq m\leq n$. Alikhani et al. \cite[Lemma 1]{Alikhani} proved that $CC(G)=1$ if and only if $G=K_{1}$ for any graph $G$. Note that if there is no connected coalition partition for a graph $G$, then $CC(G)=0$. Let $\mathcal{F}$ be a family of graphs $H$ satisfying that the subgraph obtained by removing all full vertices from $H$ is not connected. Alikhani et al. \cite[Theorem 10]{Alikhani} obtained that $CC(G)=0$ if and only if $G\in \mathcal{F}$. Hence, the following statement also holds.

\begin{theorem}\emph{\cite[Theorem 6]{Alikhani}}\label{connected}
If $G$ is a connected graph of order $n\geq 2$ with no full vertex, then $CC(G)\geq 2$.
\end{theorem}

Alikhani et al. \cite{Alikhani} also proved that $CC(G)\geq 2d_{c}(G)$ for any connected graph $G$ of order $n$ with no full vertex, and provided two polynomial-time algorithm to find graphs $G$ with $CC(G)=n-1$ and $CC(G)=n$. For a tree $T$ with order $n$, it is clear that if $n=1$, then $CC(T)=1$; if $n=2$, then $CC(T)=2$. Moveover, if $n\geq 3$ and there is a full vertex in $T$, then $T\in \mathcal{F}$ and hence $CC(T)=0$.
\begin{theorem}\emph{\cite[theorem 17]{Alikhani}}\label{tree}
For any tree $T$ with no full vertex, we have $CC(T)=2$.
\end{theorem}

In this paper, we give a brief proof of Theorem~\ref{tree} by proving the following result in Section 2.

\begin{theorem}\label{cut-set}
Let $G$ be a connected graph with no full vertex. Let $X=\{v\in V(G)\mid G-v \ \text{is not connected~}\}$. Then $CC(G)=2$ if and only if $X$ is a connected dominating set of $G$.
\end{theorem}


The \emph{corona product} of two graphs $G$ and $H$, denoted by $G\circ H$, is defined as the graph obtained by taking one copy of $G$ and $|V(G)|$ copies of $H$ and joining the $i$-th vertex of $G$ to every vertex of the $i$-th copy of $H$. Alikhani et al. \cite{Alikhani} determined the connected coalition number of $G\circ K_{1}$ for any connected graph $G$.

\begin{theorem}\emph{\cite[theorem 15]{Alikhani}}\label{corona}
$CC(G\circ K_{1})=2$ for any connected graph $G$.
\end{theorem}

Alikhani et al. \cite{Alikhani} posed the following question.

\begin{question}
What is the connected coalition number of the corona product, the join, the Cartesian product and the lexicographical product of two graphs?
\end{question}

By Theorem \ref{cut-set}, we obtain the connected coalition number of the corona product of two graphs,  which generalizes Theorem \ref{corona}.

\begin{corollary}\label{corona2}
Let $G$ be a connected graph. Then for any graph $H$, we have
\begin{equation*}
CC(G\circ H)=\left\{\begin{array}{ll}
2, &\text{if $|V(G)|\geq 2$},\\
0,&\text{if $|V(G)|=1$ and $CC(H)=0$},\\
1+CC(H),&\text{if $|V(G)|=1$ and $CC(H)\neq 0$}.
\end{array}\right.
\end{equation*}
\end{corollary}

The \emph{join} of two graphs $G$ and $H$, denoted by $G\vee H$, is defined as the graph formed by connecting every vertex of $G$ and every vertex of $H$ from disjoint copies $G$ and $H$.

\begin{theorem}\label{join}
Let $G$ and $H$ be two graphs. Then
\begin{equation*}
CC(G\vee H)=\left\{\begin{array}{ll}
|V(G)|+|V(H)|, &\text{if neither $G$ nor $H$ are complete graphs},\\
&\text{if one of $G$ and $H$ is a complete graph }\\
0,& \text{and another has connected coalition}\\
&\text{number $0$},\\
CC(G)+CC(H),&\text{others}.
\end{array}\right.
\end{equation*}
\end{theorem}

\begin{figure}
\centering
\includegraphics[width=0.6\textwidth]{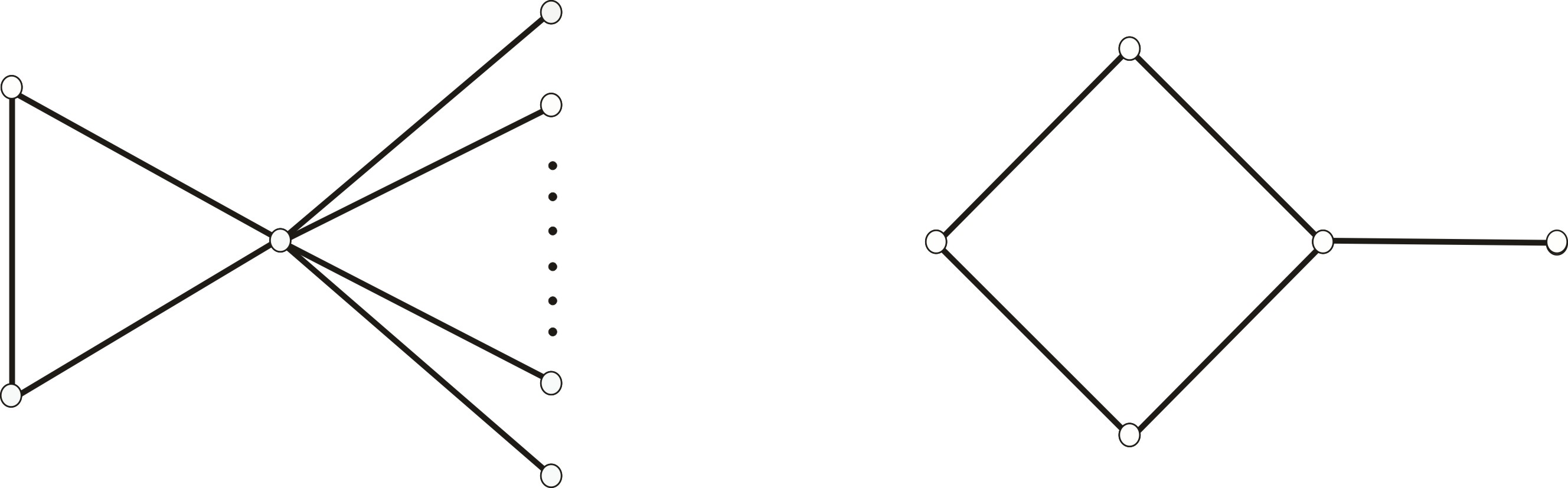}
\caption{(a) The family $\mathcal{G}$.~~~~~~~~~~~~~~~~~~~~(b) $C_{4}+e$.~~~~~~~~~~~~~~~~~}
\label{K3}
\end{figure}

We study the connected coalition number of unicycle graphs in Section 3. A family $\mathcal{G}$ of graphs is constructed as follows: the graphs obtained by identifying a vertex of $K_{3}$ and the full vertex of star graphs,  see Figure~\ref{K3} (a).

\begin{theorem}\label{unicycle graph}
Let $G$ be an unicycle graph of order $n$ with the cycle $C_{m}$, and let $Y=\{v\in V(C_{m})\mid G-v \ \text{is connected~}\}$. Then
\begin{equation*}
CC(G)=\left\{\begin{array}{ll}
4, &\text{if $G=C_{4}$},\\
0,&\text{if $G\in \mathcal{G}$},\\
2, &\text{if $n\geq 5$ and $|Y|\leq 1$ or $G[Y]=K_{2}$},\\
3,&\text{others}.
\end{array}\right.
\end{equation*}
\end{theorem}

Further, in Section 4 of this paper, we provide a lower bound for the connected coalition number of the Cartesian product and the lexicographical product of two graphs.

\section{Proofs of Theorems \ref{tree}, \ref{cut-set}, \ref{join} and Corollary \ref{corona2}}
\noindent

In this section, we give proofs of Theorems \ref{cut-set} and \ref{join}. Moreover, we give a proof of Corollary \ref{corona2} and provide a brief proof of Theorem \ref{tree} by using Theorem \ref{cut-set}.

Let $G$ be a graph of order $n$ with $CC(G)\geq 1$. If $deg(v)=n-1$ for some vertex $v\in V(G)$, then $\{v\}\in \Phi$ for any $CC(G)$-partition $\Phi$. We begin our proof with the following observation.

\begin{observation}\label{full-vertex}
Let $G$ be a connected graph with a full vertex $v$,  and let $H=G-v$. Then
\begin{equation*}
CC(G)=\left\{\begin{array}{ll}
0,&\text{if $CC(H)=0$},\\
1+CC(H),&\text{if $CC(H)\neq 0$}.
\end{array}\right.
\end{equation*}
\end{observation}

Now, we give a proof of Theorem \ref{join} by Observation \ref{full-vertex}.

\textbf{Proof of Theorem \ref{join}:} Assume first that neither $G$ nor $H$ are complete graphs.  Let $\Phi$ be a partition of  $V(G\vee H)$ such that each vertex forms a set of $\Phi$. Further, we take $$V_{1}(G)=\{v\in V(G)\mid v \ \text{is not a full vertex in}\ G\}$$ and
$$V_{1}(H)=\{v\in V(H)\mid v \ \text{is not a full vertex in}\ H\}.$$
Then $V_{1}(G)\neq \emptyset$ and $V_{1}(H)\neq \emptyset$. It is easy to see that $\{v\}$ and $\{w\}$ forms a connected coalition of $G\vee H$ for any $v\in V_{1}(G)$ and $w\in V_{1}(H)$. Note that $u$ is a full vertex in $G\vee H$ for all $u\in V(G\vee H)\backslash  (V_{1}(G)\cup V_{1}(H))$. This implies that $\Phi$ is a connected coalition partition of $G\vee H$. Therefore, $CC(G\vee H)=|V(G)|+|V(H)|$.

Further, assume that there is at least one complete graph  in $G$ and $H$. Recall that the connected coalition number of a complete graph is the number of its vertex set. Therefore, the conclusion holds by Observation \ref{full-vertex}. This completes the proof.

Next, we focus on connected coalition partitions of graphs with cut vertices.

\begin{lemma}\label{cut-vertex}
Let $G$ be a graph and $\Phi$ be a $CC(G)$-partition of $G$. If $A\in \Phi$ and $B\in \Phi$ form a connected coalition of $G$, then $v\in A$ or $v\in B$ for every cut vertex $v$ of $G$.
\end{lemma}

\begin{proof}
Suppose to the contrary that $v\notin A$ and $v\notin B$ for some cut vertex $v$ of $G$. Let $G_{1},G_{2},\ldots,G_{k}$ ($k\geq 2$) be the connected components of $G-v$. If there is a connected component $G_{i}$ with $1\leq i\leq k$ such that $A\cup B\subseteq V(G_{i})$, then the vertices in $\cup _{j\neq i}V(G_{j})$ are not dominated by $A\cup B$. This contradicts that $A$ and $B$ form a connected coalition of $G$. Otherwise, $G[A\cup B]$ is not connected, which again contradicts that $A$ and $B$ form a connected coalition of $G$. Therefore, $v\in A$ or $v\in B$. This completes the proof.
\end{proof}

\begin{lemma}\label{cut-vertices}
Let $G$ be a connected graph of $CC(G)\geq 3$ with no full vertex and let $\Phi$ be a $CC(G)$-partition of $G$. Then $v$ and $w$ belong to the same set in $\Phi$ for any two distinct cut vertices $v$ and $w$ of $G$.
\end{lemma}

\begin{proof}
Since $G$ is a connected graph with no full vertex and $CC(G)\geq 3$, there is a set $A\in \Phi$ such that $v\notin A$ and $w\notin A$. Further, there is a set $B\in \Phi$ such that $A$ and $B$ form a connected coalition of $G$. Therefore, by Lemma~\ref{cut-vertex}, $v\in B$ and $w\in B$. This completes the proof.
\end{proof}

Finally, we give an observation that will be useful later.

\begin{observation}\label{dominating-set}
Let $G$ be a connected graph with no full vertex, and let $A\subseteq V(G)$ with $|A|\geq 2$ be a connected dominating set of $G$. Then there is a partition $\Phi=\{A_{1},A_{2},\ldots,A_{k}\}$ of $A$ such that for any $A_{i}\in \Phi$, $A_{i}$ and $A_{j}$ form a connected coalition of $G$ for some $A_{j}\in \Phi$.
\end{observation}

\begin{proof}
Let $X\subseteq A$ be a minimal connected dominating set of $G$, that is, $X'$ is not a connected dominating set of $G$ for any proper subset $X'\subseteq X$. Note that $|X|\geq 2$ due to no full vertex of $G$. Then $X_{1}$ and $X_{2}$ form a connected coalition of $G$ for any partition $\{X_{1},X_{2}\}$ of $X$, in which $|X_{1}|\geq1$ and $|X_{2}|\geq1$. If $X=A$, then we are done. Thus, we consider that $A\backslash X\neq \emptyset$.

Let $A\backslash X=\{x_{1},x_{2},\ldots ,x_{s}\}$ and $Y_{r}=X\cup \{x_{1},x_{2},\ldots ,x_{r}\}$ for any $r\leq s$. Clearly, if $r=0$, then $Y_{r}=X$. Assume that there is a partition $\Phi_{r}=\{A_{1},A_{2},\ldots,A_{t}\}$ of $Y_{r}$ such that for any $A_{i}\in \Phi_{r}$, $A_{i}$ and $A_{j}$ form a connected coalition of $G$ for some $A_{j}\in \Phi_{r}$. If $\{x_{r+1}\}\cup A_{i}$ is a connected dominating set of $G$ for some $i\in \{1,2,\ldots,t\}$, then let $\Phi_{r+1}=\{A_{1},A_{2},\ldots,A_{t},\{x_{r+1}\}\}$, otherwise let $\Phi_{r+1}=\{A_{1}\cup\{x_{r+1}\},A_{2},\ldots,A_{t}\}$. It is easy to see that $\Phi_{r+1}$ is a partition of $Y_{r+1}$ such that for any $A_{i}'\in \Phi_{r+1}$, $A_{i}'$ and $A_{j}'$ form a connected coalition of $G$ for some $A_{j}'\in \Phi_{r+1}$. Following this step for all vertices in $\{x_{1},x_{2},\ldots ,x_{s}\}$ until $r=s$, we can obtain a partition $\Phi=\{A_{1},A_{2},\ldots,A_{k}\}$ of $A$ such that for any $A_{i}\in \Phi$, $A_{i}$ and $A_{j}$ form a connected coalition of $G$ for some $A_{j}\in \Phi$. This completes the proof.
\end{proof}


\textbf{Proof of Theorem \ref{cut-set}:} We first prove the sufficiency. It is clear that $CC(G)\geq2$ by Theorem \ref{connected}. Assume that $CC(G)\geq3$. Let $\Phi$ be a $CC(G)$-partition of $G$. By Lemma~\ref{cut-vertices}, there is a set $A\in \Phi$ such that $X\subseteq A$. Then $A$ is a connected dominating set of $G$ since $X$ is a connected dominating set of $G$. This contradicts that $\Phi$ is a connected coalition partition of $G$. Hence, $CC(G)=2$.

Next, we prove the necessity. Suppose to the contrary that $X$ is not a connected dominating set of $G$. Let $Y$ be a minimal connected dominating set of $G$ with $X\subseteq Y$. Then $Y\backslash X\neq \emptyset$. If $V(G)\backslash Y$ is not a connected dominating set of $G$, then let $\Phi=\{Y\backslash \{v\}, \{v\}, V(G)\backslash Y\}$ for some $v\in Y\backslash X$. Since $v\notin X$, $(Y\backslash \{v\})\cup(V(G)\backslash Y)=V(G)\backslash\{v\}$ is a connected dominating set of $G$. This implies that $\Phi$ is a connected coalition partition of $G$. Therefore, $CC(G)\geq 3$, a contradiction. Assume that $V(G)\backslash Y$ is a connected dominating set of $G$. Since $G$ has no full vertex, $|V(G)\backslash Y|\geq 2$. By Observation \ref{dominating-set}, we know that there is a partition $\{Y_{1},Y_{2},\ldots,Y_{k}\}$ of $V(G)\backslash Y$ such that for any $Y_{i}$, $Y_{i}$ and $Y_{j}$ form a connected coalition of $G$ for some $Y_{j}\in \{Y_{1},Y_{2},\ldots,Y_{k}\}$. This implies that $\Phi=\{Y\backslash \{v\}, \{v\}, Y_{1},\ldots,Y_{k}\}$ is a connected coalition partition of $G$ for some $v\in Y$. Therefore, $CC(G)\geq 4$, again a contradiction. This proves Theorem~\ref{cut-set}.

We close this section with a brief proof of Theorem \ref{tree} and Corollary  \ref{corona2} by using Theorem \ref{cut-set}.

\textbf{Proof of Theorem \ref{tree}:} Let $X=\{v\in V(T)\mid T-v \ \text{is not connected}\}$, that is, $X$ contains all of vertices other than pendant vertices of $T$. Clearly, $X$ is a connected dominating set of $T$. Therefore, by Theorem \ref{cut-set}, we have $CC(T)=2$. This completes the proof.

\textbf{Proof of Corollary \ref{corona2}:} Assume that $V(G)=\{v\}$. Then $v$ is a full vertex of $G\circ H$. By Observation \ref{full-vertex}, $CC(G\circ H)=0$ if $CC(H)=0$ and $CC(G\circ H)=1+CC(H)$ if $CC(H)\neq0$.

We now need only to consider that $|V(G)|\geq 2$. It is easy to see that $G\circ H$ has no full vertex. Let $X=\{v\in V(G\circ H)\mid G\circ H-v \ \text{is not connected}\}$. Obviously, $X=V(G)$ and $X$ is a connected dominating set of $G\circ H$. Thus, $CC(G\circ H)=2$ by Theorem \ref{cut-set}. This completes the proof.

\section{Proof of Theorem \ref{unicycle graph}}
\noindent

In this section, we study the connected coalition number of unicycle graphs by proving Theorem \ref{unicycle graph}. We start with the connected coalition number of cycles.

\begin{lemma} \label{cycle}
For any cycle $C_{n}$ with order $n$, we have
\begin{equation*}
CC(C_{n})=\left\{\begin{array}{ll}
4, &\text{if $n=4$},\\
3,&\text{otherwise}.
\end{array}\right.
\end{equation*}
\end{lemma}

\begin{proof}
Let $C_{n}=v_{1}v_{2}\cdots v_{n}v_{1}$. It is easy to check that $CC(C_{3})=3$ and $CC(C_{4})=4$. Thus, we assume that $n\geq 5$. It is not hard to see that $\{\{v_{2}\},\{v_{n}\},V(C_{n})\backslash \{v_{2},v_{n}\}\}$ is a connected coalition partition of $G$. Hence, $CC(C_{n})\geq 3$. We now need only to prove that $CC(C_{n})\leq3$. Suppose to the contrary that $CC(C_{n})\geq 4$. Let $\Phi$ be a $CC(C_{n})$-partition of $C_{n}$.

Note that the subgraph induced by a connected dominating set of $C_{n}$ is either the cycle $C_{n}$ or a path $P_{n-1}$ or a path $P_{n-2}$. Since $CC(C_{n})\geq 4$, for any two sets in $\Phi$, say $A$ and $B$, we have $C_{n}[A\cup B]=P_{n-2}$. Without loss of generality, we assume that $P_{n-2}=v_{1}v_{2}\cdots v_{n-2}$. In this way, $\{v_{n-1}\}\in \Phi$ and $\{v_{n}\}\in \Phi$ since $CC(C_{n})\geq 4$. Note that $\{v_{n-1},v_{n}\}$ is not a dominating set of $C_{n}$. Then either $\{v_{n-1}\}\cup A$ or $\{v_{n-1}\}\cup B$ is a connected dominating set of $C_{n}$. This implies that $\Phi=\{\{v_{1}\},\{v_{2},v_{3},\ldots,v_{n-2}\},\{v_{n-1}\},\{v_{n}\}\}$. However, none of $\{v_{1},v_{n}\}$, $\{v_{2},v_{3},\ldots,v_{n-2},v_{n}\}$ and $\{v_{n-1},v_{n}\}$ is a connected dominating set of $C_{n}$, which contradicts that $\Phi$ is a connected coalition partition of $C_{n}$. Therefore, $CC(C_{n})\leq3$ and so $CC(C_{n})=3$. This completes the proof.
\end{proof}

We now discuss about the relation of the connected coalition number between graphs $G$ with pendant vertices $X$ and graphs $G-X$.

\begin{lemma}\label{degree 1}
Let $H$ be a connected graph of order $n\geq 3$ with no full vertex. If $G$ is a graph obtained by identifying a vertex of $H$ and a vertex of $K_{2}$, then $CC(G)\leq CC(H)$.
\end{lemma}

\begin{proof}
 By Theorem \ref{connected}, we know that $CC(G)\geq 2$. Let $v$ be the pendant vertex of $G$ that comes from $K_{2}$, and $w$ be the neighborhood of $v$ in $G$. Let $\Phi=\{A_{1},A_{2},\ldots,A_{CC(G)}\}$ be a $CC(G)$-partition of $G$ that satisfies $w\in A_{1}$ and $|A_{1}|$ is maximum, that is, for every $CC(G)$-partition $\{B_{1},B_{2},\ldots,B_{CC(G)}\}$ of $G$, if $w\in B_{1}$, then $|A_{1}|\geq|B_{1}|$. We say that $\{v\}\notin \Phi$. If not, then $\{v\}$ and $A_{1}$ form a connected coalition of $G$ by Lemma \ref{cut-vertex}. This implies that $A_{1}$ is a connected dominating of $G$, which contradicts that $\Phi$ is a connected coalition partition of $G$.

 Let $\Phi'=\{A'_{1},A'_{2},\ldots,A'_{CC(G)}\}$, where $A'_{i}=A_{i}\backslash\{v\}$ for all $i=\{1,2,\ldots,$ $CC(G)\}$. Since $A_{1}$ is not a connected dominating set of $G$, $A'_{1}$ is not a connected dominating set of $H$. Define
 $$I=\{i\in\{2,3,\ldots,CC(G)\}\mid A'_{i} \ \text{is not a connected dominating set of} \ H\}.$$
 We separate the proof into two cases.

{\bf Case 1.} $I\neq \emptyset$.

Let $\{A'_{j_{1}},A'_{j_{2}},\ldots,A'_{j_{k}}\}$ be a partition of $A'_{j}$ that satisfies the condition in Observation \ref{dominating-set} for all $j\in \{2,3,\ldots,CC(G)\} \backslash I$. A partition $\Psi$ of $V(H)$ is constructed as follows:

(i) $A'_{1}\in \Psi$;

(ii) $A'_{i}\in \Psi$ for all $i\in I$;

(iii) $\{A'_{j_{1}},A'_{j_{2}},\ldots,A'_{j_{k}}\} \subseteq \Psi$ for all $j\in \{2,3,\ldots,CC(G)\} \backslash I$.

Recall that $w\in A_{1}$ and $w$ is a cut vertex of $G$. By Lemma~\ref{cut-vertex}, $A_{1}$ and $A_{i}$ form a connected coalition of $G$ for all $i\in \{2,3,\ldots,CC(G)\}$. Then $A'_{1}$ and $A'_{j}$ form a connected coalition of $H$ for all $j\in I$. Therefore, $\Psi$ is a connected coalition partition of $H$, and so $CC(G)\leq |\Psi|\leq CC(H)$.

{\bf Case 2.} $I=\emptyset$.

Recall that $A'_{1}$ is not a connected dominating set of $H$. We now divided the proof into two subcases.
\begin{enumerate}
\item[\textbf{(2-1)}]  $A'_{1}\cup \{u\}$ is a connected dominating set of $H$ for some $u\in V(H)\backslash A'_{1}$.
 \end{enumerate}

In this case, without loss of generality, we assume that $u\in A'_{2}$. Let $\{A'_{j_{1}},A'_{j_{2}},\ldots,A'_{j_{k}}\}$ is a partition of $A'_{j}$ that satisfies the condition in Observation \ref{dominating-set} for all $j\in \{3,\ldots,CC(G)\}$. A partition $\Psi$ of $V(H)$ is constructed as follows:

 (i) $A'_{1}\in \Psi$ and $\{u\}\in \Psi$;

 (ii) $\{A'_{j_{1}},A'_{j_{2}},\ldots,A'_{j_{k}}\}\subseteq \Psi$ for all $j\in \{3,\ldots,CC(G)\}$;

 (iii) If $A'_{2}\backslash \{u\}$ is not a connected dominating set of $H$, then we take $A'_{2}\backslash \{u\}\in \Psi$. If $A'_{2}\backslash \{u\}$ is a connected dominating set of $H$, then we take $\{A'_{2_{1}},A'_{2_{2}},\ldots,A'_{2_{k}}\}\subseteq \Psi$, where $\{A'_{2_{1}},A'_{2_{1}},\ldots,A'_{2_{k_{2}}}\}$ is a partition of $A'_{2}\backslash \{u\}$ that satisfies the condition in Observation~\ref{dominating-set}.

It is obvious that $\Psi$ is a connected coalition partition of $H$. Therefore, $CC(G)\leq |\Psi|\leq CC(H)$.

\begin{enumerate}
\item[\textbf{(2-2)}]  $A'_{1}\cup \{u\}$ is not a connected dominating set of $H$ for all vertex $u\in V(H)\backslash A'_{1}$.
 \end{enumerate}

Since $H$ has no full vertex and $I=\emptyset$, $|A'_{2}|\geq 2$. For a vertex $u\in A'_{2}$, a partition $\Theta$ of $V(G)$ is constructed as follows:

(i) $A_{1}\cup \{u\}\in \Theta$;

(ii) $A_{2}\backslash \{u\}\in \Theta$;

(iii) $A_{i}\in \Theta$ for all $i\in \{3,\ldots,CC(G)\}$.

Since $A'_{1}\cup \{u\}$ is not a connected dominating set of $H$, $A_{1}\cup \{u\}$ is not a connected dominating set of $G$. By Lemma~\ref{cut-vertex}, $A_{1}$ and $A_{i}$ form a connected coalition of $G$ for all $i\in \{2,3,\ldots,CC(G)\}$. Therefore, $\Theta$ is a $CC(G)$-partition of $G$. However, $|A_{1}\cup \{u\}|\geq |A_{1}|$, which contradicts the choice of the set $A_{1}$. This proves Lemma~\ref{degree 1}.
\end{proof}

\textbf{Proof of Theorem \ref{unicycle graph}:} By Lemma~\ref{cycle} and Observation~\ref{full-vertex}, the conclusion holds for $G=C_{3}$, $G=C_{4}$ and $G\in \mathcal{G}$. Thus, we assume that $G$ has no full vertex and $n\geq 5$.

Let $X=\{v\in V(G)\mid G-v\ \text{is not connected}\}$ and $Z=\{v\in V(G)\mid deg(v)=1\}$. Then $V(G)=X\cup Y\cup Z$. It is easy to see that if $|Y|\leq 1$ or $G[Y]=K_{2}$, then $X$ is a connected dominating set of $G$. Therefore, $CC(G)=2$ by Theorem \ref{cut-set}.


We now need only to consider that $|Y|\geq 3$ and $Y$ consists of two non-adjacent vertices of $C_{m}$. Let $C_{m}=v_{1}v_{2}\cdots v_{m}v_{1}$. It is obvious that $m\geq 4$ and if $|Y|\geq 3$, then $Y$ contains two non-adjacent vertices of $C_{m}$. Without loss of generality, we assume that $v_{i},v_{j}\in Y$ and $v_{i}v_{j}\notin E(G)$ for some $i,j\in \{1,2,\ldots,m\}$. Let $C_{4}+e$ be the graph obtained by identifying a vertex of $C_{4}$ and a vertex of $K_{2}$, see Figure~\ref{K3} (b). It is not hard to check that $CC(C_{4}+e)=3$. By Lemmas~\ref{cycle} and \ref{degree 1}, $CC(G)\leq CC(C_{m})=3$ if $m\geq 5$ and $CC(G)\leq CC(C_{4}+e)=3$ if $m=4$. On the other hand, note that $G-\{v_{i},v_{j}\}$ is not connected. Therefore, $\{\{v_{1}\},\{v_{k}\},V(G)\backslash\{v_{1},v_{k}\}\}$ is a connected coalition partition of $G$ and so $CC(G)=3$. This completes the proof.

\section{Lower bound of the connected coalition number of products of two graphs}
\noindent

The \emph{Cartesian product} of two graphs $G$ and $H$, denoted by $G\Box H$, is defined as the vertex set $V(G)\times V(H)=\{(u,v)\mid u\in V(G),v\in V(H)\}$ with an edge between vertices $(u_{1},v_{1})$ and $(u_{2},v_{2})$ if either $v_{1}$ is adjacent to $v_{2}$ in $H$ and  $u_{1}=u_{2}$, or $u_{1}$ is adjacent to $u_{2}$ in $G$ and  $v_{1}=v_{2}$.

\begin{theorem}\label{Cartesian}
Let $G$ and $H$ be two connected graphs with at least two vertices. Then $CC(G\Box H)\geq \max\{CC(G)+k_{G},CC(H)+k_{H}\}$, where $k_{G}$ and $k_{H}$ denote the number of full vertices in $G$ and $H$, respectively.
\end{theorem}
\begin{proof}
Without loss of generality, we assume that $CC(G)+k_{G}\geq CC(H)+k_{H}$. Since $G$ has at least two vertices, $G\Box H$ has no full vertex. Moreover, since $G$ and $H$ are two connected graphs, $G\Box H$ is also connected. Let $u_{1},u_{2},\ldots,u_{k_{G}}$ be all of the full vertices of $G$.

We first consider that $CC(G)=0$. Since $G$ is a connected graph, $k_{G}\geq 1$. Let $P_{i}=\{(u,v)\in V(G)\times V(H)\mid u=u_{i}, v\in V(H)\}$ for $i\in \{1,2,\ldots,k_{G}-1\}$ and $P_{k_{G}}=V(G\Box H)\setminus(\cup_{i=1}^{k_{G}-1}V(P_{i}))$. Then $P_{i}$ is a connected dominating set of $G\Box H$ for every $i\in \{1,2,\ldots,k_{G}\}$. Further, since $H$ has at least two vertices, $|P_{i}|\geq 2$. Therefore, we can obtain a connected coalition partition of $G\Box H$ with the cardinality at least $2k_{G}$ by Observation~\ref{dominating-set}. Hence, $CC(G\Box H)\geq2k_{G}\geq CC(G)+k_{G}$.

Recall that $CC(G)=1$ if and only if $G=K_{1}$ for any graph $G$. Thus, we now need only to consider that $CC(G)\geq 2$. Let $\Phi=\{A_{1},A_{2},\ldots,A_{CC(G)}\}$ be a $CC(G)$-partition of $G$ and $Q_{i}=\{(u,v)\in V(G)\times V(H)\mid u\in A_{i},v\in V(H)\}$ for all $i\in \{1,2,\ldots,CC(G)\}$. Without loss of generality, we assume that $A_{i}=\{u_{i}\}$ for all $i\in \{1,2,\ldots,k_{G}\}$. Note that $Q_{i}$ is a connected dominating set of $G\Box H$ for all $i\in \{1,\ldots,k_{G}\}$. Moreover, $|B_{i}|\geq 2$ due to $|V(G)|\geq 2$. This implies that there exists a partition $\{Q_{i_{1}},Q_{i_{2}},\ldots,Q_{i_{k_{i}}}\}$ of $Q_{i}$ satisfying the condition in Observation~\ref{dominating-set}. A partition $\Psi$ of $V(G\Box H)$ is constructed as follows:

(i) $Q_{i}\in \Psi$ for all $i\in \{k_{G}+1,k_{G}+2,\ldots,CC(G)\}$;

(ii) $\{Q_{i_{1}},Q_{i_{2}},\ldots,Q_{i_{k}}\} \subseteq \Psi$ for all $i\in \{1,2,\ldots,k_{G}\}$.

Note that for any $A_{i}\in \Phi$, there exists an $A_{j}\in \Phi$ such that $A_{i}$ and $A_{j}$ form a connected coalition of $G$, where $i,j\in \{k_{G}+1,k_{G}+2,\ldots,CC(G)\}$ and $i\neq j$. Therefore, $Q_{i}$ and $Q_{j}$ also form a connected coalition of $G\Box H$. This implies that $\Psi$ is a connected coalition partition of $G\Box H$. Hence, $CC(G\Box H)\geq|\Phi|\geq (CC(G)-k_{G})+2k_{G} \geq CC(G)+k_{G}$. This completes the proof.
\end{proof}

We next improve the lower bound in Theorem \ref{Cartesian} for the connected coalition number of the Cartesian product of two special graphs.

\begin{theorem}
Let $G$ and $H$ be two graphs with order $n\geq 2$ and $m\geq 2$, respectively. If there is a full vertex $u$ in $G$ such that $G-u$ is connected, and there is a full vertex $v$ in $H$ such that $H-v$ is also connected, then $CC(G\Box H)\geq m+n-1$.
\end{theorem}
\begin{proof}
Let $V(G)=\{u_{1},u_{2},\ldots,u_{n}\}$ and $V(H)=\{v_{1},v_{2},\ldots,v_{m}\}$. Without loss of generality, we assume $u=u_{1}$ and $v=v_{1}$. Let $A_{1}=\{(x,y)\in V(G)\times V(H)\mid x=u_{1} \ \text{and} \ y=v_{1}, \ \text{or}\ x\neq u_{1} \ \text{and} \ y\neq v_{1}\}$, $A_{i}=\{(u_{i},v_{1})\}$ for all $i=2,3,\ldots,n$, and $A_{i}=\{(u_{1},v_{i-n+1})\}$ for all $i=n+1,n+2,\ldots,m+n-1$. Then $A_{1}$ and $A_{i}$ form a connected coalition of $G\Box H$ for all $i=2,3,\ldots,m+n-1$. This implies that $\{A_{1},A_{2},\ldots,A_{m+n-1}\}$ is a connected coalition partition of $G\Box H$. Hence, $CC(G\Box H)\geq m+n-1$.
\end{proof}

The \emph{lexicographic product} of two graphs $G$ and $H$, denoted by $G\circ H$, is defined as the vertex set $V(G)\times V(H)$ with an edge between vertices $(u_{1},v_{1})$ and $(u_{2},v_{2})$ if either $v_{1}$ is adjacent to $v_{2}$ in $H$ and  $u_{1}=u_{2}$, or $u_{1}$ is adjacent to $u_{2}$ in $G$.

\begin{theorem}
Let $G$ and $H$ be two graphs with at least two vertices. Then $CC(G\circ H)\geq CC(G)+k_{G}$, where $k_{G}$ is the number of full vertices in $G$.
\end{theorem}

\begin{proof}
Clearly, $CC(G)\neq 1$.
Similar to the proof in Theorem \ref{Cartesian}, we can obtained a connected coalition partition of the cardinality at least $2k_{G}$ for $CC(G)=0$ and the cardinality at least $CC(G)+k_{G}$ for $CC(G)\geq 2$, respectively.
\end{proof}


\section*{Acknowledgements}
\noindent

This work was supported by the National Natural Science Foundation of China [No. 12301455] and Natural Science Foundation of Shandong of China [No. ZR2023QA080].
\section*{References}
\bibliographystyle{model1b-num-names}
\bibliography{<your-bib-database>}

\begin{thebibliography}{9}
\bibitem{Alikhani1} S. Alikhani, D. Bakhshesh and H.R. Golmohammadi, Introduction to total coalitions in graphs. arXiv:2211.11590

\bibitem{Alikhani} S. Alikhani, D. Bakhshesh, H.R. Golmohammadi and E.V. Konstantinova, Connected coalition in graphs, Discuss. Math. Graph Theory, inpress. https://doi.org/10.7151/dmgt.2509

\bibitem{Alikhani2} S. Alikhani, H. Golmohammadi and E.V. Konstantinova, Coalition of cubic graphs of order at most 10, Commun. Comb. Optim, inpress.
https://doi.org/10.22049/cco.2023.28328.1507

\bibitem{Bakhshesh} D. Bakhshesh and M.A. Henning and D. Pradhan, On the coalition number of trees, Bull. Malays. Math. Sci. Soc. 46 (2023) 95. 

\bibitem{Barat} J. Bar\'{a}t, Z.L. Bl\'{a}zsik, General sharp upper bounds on the total coalition number, Discuss. Math. Graph Theory, inpress.  https://doi.org/10.7151/dmgt.2511

\bibitem{bondy} J.A. Bondy and U.S.R. Murty, Graph theory with application,  North-Holland, New York, 1976.

\bibitem{Cockayne} E.J Cockayne and S.T. Hedetniemi, Towards a theory of domination in graphs, Networks 7 (1977) 247-261. 

\bibitem{Du} D.Z. Du and P.J. Wan, Connected Dominating Set: Theory and Applications, Springer, New York, 2013. 

\bibitem{Goddard} W. Goddard and M.A. Henning, Independent domination in graphs: A survey and recent results, Discrete Math. 313 (2013) 839-854.

\bibitem{Hartnell} B.L. Hartnell and D.F. Rall, Connected domatic number in planar graphs, Czechoslovak Math. J. 51 (2001) 173-179.

\bibitem{Haynes} T.W. Haynes, J.T. Hedetniemi, S.T. Hedetniemi, A.A. McRae and R. Mohan, Introduction to coalitions in graphs, AKCE Int. J. Graphs Comb. 17 (2020) 653-659. 

\bibitem{Haynes4} T.W. Haynes, J.T. Hedetniemi, S.T. Hedetniemi, A.A. McRae and R. Mohan, Coalition graphs of paths, cycles and trees, Discuss. Math. Graph Theory, inpress. https://doi.org/10.7151/dmgt.2416

\bibitem{Haynes5} T.W. Haynes, J.T. Hedetniemi, S.T. Hedetniemi, A.A. McRae and R. Mohan, Upper bounds on the coalition number, Australas. J. Combin. 80 (2021) 442-453.


\bibitem{Haynes1} T.W. Haynes, S.T. Hedetniemi and M.A. Henning, Topics in Domination in Graphs, Dev. Math. 64 Springer, Cham, 2020.


\bibitem{Haynes2} T.W. Haynes, S.T. Hedetniemi and P.J. Slater, Fundamentals of Domination in Graphs, Boca Raton, CRC Press, 1998. 

\bibitem{Haynes3} T.W. Haynes, S.T. Hedetniemi and P.J. Slater, Domination in Graphs: Advanced Topics, Marcel Dekker, New York, 1998.

\bibitem{Henning0} M.A. Henning, A survey of selected recent results on total domination in graphs,
Discrete Math. 309 (2009) 32-63. 

\bibitem{Henning} M.A. Henning and A. Yeo, Total Domination in Graphs, Springer Monographs in Mathematics, 2013.

\bibitem{Zelinka1} B. Zelinka, Domatic number and degrees of vertices of a graph, Math. Slovaca 33 (1983) 145-147.

\bibitem{Zelinka2} B. Zelinka, On domatic numbers of graphs, Math. Slovaca 31 (1981) 91-95. 

\bibitem{Zelinka} B. Zelinka, Connected domatic number of a graph, Math. Slovaca 36 (1986) 387-392. 
\end{thebibliography}

\end{document}